\documentclass{article}

\usepackage{theorem}
\usepackage{enumitem,multicol}
\usepackage{amsmath,amssymb,amsfonts,graphicx,float}

\def\BBox{\rule{2mm}{3mm}}
\def\QED{\hfill$\BBox$}
\newenvironment{proof}
{\begin{rm}\par\smallskip\noindent{\bf Proof.}\quad}{\QED\end{rm}}

\theorembodyfont{\rmfamily}
\newtheorem{thm}{Theorem}[section]
\newtheorem{lem}[thm]{Lemma}        %
\newtheorem{prop}[thm]{\bfseries Proposition} %

\newtheorem{defn}[thm]{\bfseries Definition}

\begin{document}

\title{On Combinatorial Properties of Points and Polynomial Curves}

\author{
Hiroyuki Miyata
\\
Department of Computer Science, \\
Gunma University, Japan\\
{\tt hmiyata@gunma-u.ac.jp}
}

\maketitle

\begin{abstract}
Many combinatorial properties of a point set in the plane are determined by the set of possible partitions of the point set by a line. 
Their essential combinatorial properties are well captured by the axioms of oriented matroids.
In fact, Goodman and Pollack (Journal of Combinatorial Theory, Series A, Volume 37, pp.~257--293, 1984) proved that the axioms of oriented matroids of rank $3$
completely characterize the sets of possible partitions arising from a natural topological generalization of configurations of points and lines.

In this paper, we introduce a new class of oriented matroids, called degree-$k$ oriented matroids, which captures essential combinatorial properties of the possible partitions 
of point sets in the plane by the graphs of polynomial functions of degree $k$.
We prove that the axiom of degree-$k$ oriented matroids completely characterizes the sets of possible partitions arising from a natural topological generalization of configurations 
formed by points and the graphs of polynomial functions degree $k$. It turns out that the axiom of degree-$k$ oriented matroids coincides with the axiom of ($k+2$)-signotopes, which was introduced by Felsner and Weil 
(Discrete Applied Mathematics, Volume 109, pp.~67--94, 2001) in a completely different context.
Our result gives a two-dimensional geometric interpretation for ($k+2$)-signotopes and also for single element extensions of cyclic hyperplane arrangements in $\mathbb{R}^{n-k-3}$.
\end{abstract}

\section{Introduction}
Possible partitions of a point set  in the plane by a line are actively studied in discrete and computational geometry.
For example, the famous ham-sandwich theorem claims that for any set of points in the plane that are colored red or blue there is a line that simultaneously bisects the red points and the blue points (for details and related results, see~\cite{KU21}).
The number of \emph{$k$-sets} of a point set, $k$-element subsets of the point set that can be separated from the other points by a line, is also heavily studied (see \cite[Chapter 11]{M02}).
Actually, many combinatorial properties (e.g. convexity, possible triangulations) of a point set are determined by the possible partitions by a line.
If we collect all possible partitions of a point set by a line,  it corresponds to consider the \emph{order type}~\cite{GP83}, which is a fundamental notion in discrete and computational geometry.
Formally, two point sets are said to have the same order type if there is a bijection between the point sets that induces a bijection between the sets of possible partitions of the point sets by a line.
Order types have several other equivalent formulations (e.g.  by the orientation of each triple of points)~\cite{GP84} and are very useful to identify and classify combinatorial structures of point sets in the plane~\cite{AAK02,GP80,GP83,GP84}.
Furthermore, order types have a nice combinatorial abstraction, called \emph{oriented matroids}. 

Oriented matroids are a well-studied combinatorial model, introduced by Bland and Las Vergnas~\cite{BL78} and Folkman and Lawrence~\cite{FL78}.
Oriented matroids have several equivalent axiom systems, which capture common combinatorial structures that exist in point sets, polytopes, hyperplane arrangements, digraphs, etc. 
The axioms of oriented matroids are simple, but many combinatorial behaviors of those objects  are well explained by the axioms of oriented matroids.
(For example, the equivalence of the different formulations for order types can be understood as the equivalence of the \emph{covector, cocircuit,} and \emph{chirotope} axioms of oriented matroids~\cite{GP84}. Many other examples can be found in~\cite{BLSWZ99}.)
In the case of hyperplane arrangements, the Folkman-Lawrence topological representation theorem~\cite{FL78} gives a very nice explanation for this.
That is, the Folkman-Lawrence topological representation theorem claims that every (loop-free) oriented matroid has a topological representation as a \emph{pseudosphere arrangement}, which can be viewed as a topological generalization of a hyperplane arrangement.
This means that oriented matroids are a very precise combinatorial model for hyperplane arrangements.
An analogous result is also known for point sets in the plane;
Goodman and Pollack~\cite{GP84} proved that every (loop-free) oriented matroid of rank $3$ has a topological representation as a \emph{generalized configurations of points}, which means that the axioms of oriented matroids completely characterize the sets of possible partitions of point sets in the plane by a pseudoline.
Thus, the axioms of oriented matroids also capture essential combinatorial properties of possible partitions of point sets in the plane by a line.

In this paper, we investigate combinatorial properties of another kind of geometric partitions: possible partitions of a point set in the plane by the graph of a polynomial function of degree $k$.
Combinatorial properties of such partitions were first closely investigated by Eli\'a\v{s} and Matou\v{s}ek~\cite{EM13} to prove a new generalization of the Erd\H{o}s-Szekeres theorem.
In this paper, we discuss more closely what kind of combinatorial properties are exhibited by such partitions.
In fact, we show that the combinatorial properties observed by Eli\'a\v{s} and Matou\v{s}ek represent almost \emph{all} combinatorial properties that can be proved using some natural geometric properties.
To do so, we first review the combinatorial properties observed by Eli\'a\v{s} and Matou\v{s}ek, and then observe a slightly stronger combinatorial property.
Axiomatizing the observed combinatorial properties, we introduce a combinatorial model, called {\it degree-$k$ oriented matroids}. 
Then, we present a topological representation theorem for degree-$k$ oriented matroids. More specifically, we prove that every degree-$k$ oriented matroid has a topological representation as a {\it $k$-intersecting pseudoconfiguration of points}, 
which is a natural geometric generalization of a configuration of points and the graphs of polynomial functions of degree $k$.
This result can be viewed as an analogue of the result of Goodman and Pollack~\cite{GP84}, and implies that the axiom of degree-$k$ oriented matroids describes essential combinatorial properties of the possible partitions of point sets by the graphs of polynomial functions of degree $k$.
Coincidentally, it  turns out that the axiom of degree-$k$ oriented matroids coincides with the axiom of $(k+2)$-signotopes, introduced by Felsner and Weil~\cite{FW01},
which is closely related to higher Bruhat orders~\cite{MS89}.
Prior to Felsner and Weil~\cite{FW01}, signotopes were also studied by Ziegler~\cite{Z93} under the name of ``consistent subsets'' .
Ziegler~\cite{Z93} proved that $k$-signotopes (consistent subsets) have a geometric interpretation as single element extensions of cyclic hyperplane arrangements in $\mathbb{R}^{n-k-1}$.
Hence, our result gives a two-dimensional geometric interpretation for signotopes and also for single element extensions of cyclic hyperplane arrangements.

\subsection*{Related work}
The possible partitions of a point set in the Euclidean space by a certain family of spheres determine an oriented matroid, which is called a {\it Delaunay oriented matroid} (see \cite[Section 1.9]{BLSWZ99}).
Santos~\cite{S96} proved that partitions of a point set in the plane by spheres defined by a smooth, strictly convex distance function also fulfill the oriented matroid axioms.
It is proved in \cite{M16} that the class of partitions of point sets in the plane by a certain kind of pseudocircles coincides with the rank $4$ matroid polytopes. 

\subsection*{Notation}
Here, we summarize the notation that will be employed in this paper.
In the following, we assume that $S$ is a finite ordered set, $X$ is a sign vector on $E$, $P$ is a point in the plane, and $r$ is a positive integer.

\begin{multicols}{2}
\begin{itemize}[itemsep=0mm,leftmargin=*]
\item $\Lambda (S, r) := \{ (\lambda_1,\dots,\lambda_r) \in S^r \mid \lambda_1 < \cdots < \lambda_r \}$.
\item $\bar{\lambda} := \{ \lambda_1,\dots,\lambda_r \}$ for $\lambda = (\lambda_1,\dots,\lambda_r) \in  \Lambda (S,r)$.
\item $\lambda \setminus \{ \lambda_k\} := (\lambda_1,\dots,\lambda_{k-1},\lambda_{k+1},\dots,\lambda_r)$ 
\item[]for $\lambda = (\lambda_1,\dots,\lambda_r) \in \Lambda (S,r)$.
\item $\lambda [\lambda_i | k] := (\lambda_1,\dots,\lambda_{i-1},k,\lambda_{i+1}\dots,\lambda_r)$ 
\item[] for $\lambda = (\lambda_1,\dots,\lambda_r) \in  \Lambda (S,r)$.
\item $\bar{X} := \{ e \in E \mid X(e) \neq 0 \}$.
\item $X^+$ (resp. $X^-, X^0$)
\item[]$:= \{ e \in E \mid X(e) = +1$ (resp. $-1, 0$)$\}$.
\item $x(P)$: the $x$-coordinate of $P$.
\item $y(P)$: the $y$-coordinate of $P$.
\end{itemize}
\end{multicols}

\section{Preliminaries}
In this section, we summarize some basic facts regarding oriented matroids.
See \cite{BLSWZ99} for further details.

Let $P = \{ p_1,\dots,p_n \}$ be a point set  in $\mathbb{R}^d$ in general position (i.e., no $d+1$ points of $P$ lie on the same hyperplane),
and let $V=\{ v_1, \dots, v_n \}$ be the set of vectors in $\mathbb{R}^{d+1}$ with $v_e := (p_e^T,1)^T$.
To see how $P$ is separated by hyperplanes, let us consider
 the map $\chi_P : [n]^{d+1} \rightarrow \{ +1,-1,0 \}$ defined by
\[ \chi_P (i_1,\dots,i_{d+1}) := {\rm sign}\det(v_{i_1},\dots,v_{i_{d+1}}) \text { for $i_1,\dots,i_{d+1} \in [n]$,}\]
where ${\rm sign} (a) = + 1$ (resp. $-1, 0$) if $a > 0$ (resp. $a < 0, a = 0$).
Then, we have
\begin{align*}
&\chi_P (i_1,\dots,i_d,j)\chi_P (i_1,\dots,i_d,k) = +1 \\
&\Leftrightarrow \text{ The points $p_j$ and $p_k$ lie on the same side of the hyperplane spanned by $p_{i_1},\dots,p_{i_d}$.}
\end{align*}
The map $\chi_P$ contains rich combinatorial information regarding $P$, such as convexity, the face lattice of the convex hull, and possible combinatorial types of triangulations (see  \cite{BLSWZ99}).
The chirotope axioms of oriented matroids are obtained by abstracting the properties of  $\chi_P$.

\begin{defn}(Chirotope axioms for oriented matroids)\\
For $r \in \mathbb{N}$ and a finite set $E$, a map $\chi: E^r \rightarrow \{ +1,-1,0\}$ is called a {\it chirotope}
if it satisfies the following axioms. The pair $(E,\{ \chi, -\chi \})$ is called an {\it oriented matroid of rank $r$}.
\begin{itemize}
\item[(B1)] $\chi$ is not identically zero.
\item[(B2)] $\chi(i_{\sigma(1)},\dots,i_{\sigma(r)}) = {\rm sgn}(\sigma)\chi (i_1,\dots,i_r)$, for any $i_1,\dots,i_r \in E$ and any permutation $\sigma$ on $E$.
\item[(B3)] For any $\lambda, \mu \in E^r$, we have  
\[ \{ \chi (\lambda )\chi (\mu ) \} \cup \{ \chi (\lambda[\lambda_1 | \mu_s]) \chi (\mu [\mu_s | \lambda_1]) \mid s = 1,\dots, r \} \supset \{ +1,-1\} \text{ or } = \{ 0 \} .\]
\end{itemize}
\end{defn}
We remark that (B3) is combinatorial abstraction of the Grassmann-Pl\"ucker relations: 
\[ \det(V_{\lambda})\det(V_{\mu}) = \sum_{s=1}^r{\det(V_{\lambda[\lambda_1|\mu_s]})\det(V_{\mu[\mu_s|\lambda_1]})},   \]
where $V_{\tau} := (v_{\tau_1},\dots,v_{\tau_r})$ for $\tau \in \Lambda([n],d+1)$.

The set $C^*_P := \{ \pm (\chi_P (\lambda, e))_{e \in [n]} \mid \lambda \in [n]^d \}$ also contains equivalent information to $\chi_P$.
The cocircuit axioms of oriented matroids are introduced by abstracting the properties of the set ${\cal C}^*_P$. 
\begin{defn}(Cocircuit axioms for oriented matroids)\\
A collection ${\cal C}^* \subset \{ +1,-1,0\}^E$ satisfying Axioms (C0)--(C3)  is called  the set of {\it cocircuits} of an oriented matroid.
\begin{itemize}
\item[(C0)] $0 \notin {\cal C}^*$.
\item[(C1)] ${\cal C}^* = -{\cal C}^*$.
\item[(C2)] For all $X,Y \in {\cal C}^*$, if $\bar{X} \subset \bar{Y}$, then $X = Y$ or $X=-Y$.
\item[(C3)] For any $X,Y \in {\cal C}^*$  and $e \in (X^+ \cap Y^-) \cup (X^- \cap Y^+)$, there exists $Z \in {\cal C}^*$ with
\item[] $Z^0 = (X^0 \cap Y^0) \cup \{ e \}$, $Z^+ \supset X^+ \cap Y^+$, and  $Z^- \supset X^- \cap Y^-$.
\end{itemize}
\end{defn}
From a chirotope $\chi$, we can construct the cocircuits $C^* := \{ \pm (\chi(\lambda,e))_{e \in E} \mid \lambda \in \Lambda (E,r-1) \}$.
It is also possible to reconstruct the chirotope $\chi$ (up to a sign reversal) from the cocircuits ${\cal C}^*$.
A rank $r$ oriented matroid $(E, \{ \chi, -\chi \} )$ is said to be {\it uniform} if $\chi (\lambda ) \neq 0$ for any $\lambda \in \Lambda (E,r)$ 
(equivalently if $|X^0| = r$ for any cocircuit $X$).
If the underlying structure of the set ${\cal C}^* \subset \{ +,-,0\}^E$ is known to be uniform, i.e., if $|X^0|=r$ for any $X \in {\cal C}^*$,
then Axiom (C3) can be replaced by the following axiom:
\begin{quote}
(C3') For any $X,Y \in {\cal C}^*$ with $|X^0 \setminus Y^0| = 1$ and $e \in (X^+ \cap Y^-) \cup (X^- \cap Y^+)$, there exists $Z \in {\cal C}^*$ 
with $Z^0 = (X^0 \cap Y^0) \cup \{ e \}$, $Z^+ \supset X^+ \cap Y^+$, and  $Z^- \supset X^- \cap Y^-$.
\end{quote}
More generally, Axiom (C3) can be replaced by the axiom of  {\it modular elimination}. For further details, see \cite[Section 3.6]{BLSWZ99}.
An oriented matroid $(E,{\cal C}^*)$ is {\it acyclic} if for any $e \in E$ there exists $X_e \in {\cal C}^*$ with $e \in X_e^+$ and $X_e^- = \emptyset$.
It can easily be seen that oriented matroids arising from point configurations are acyclic.
\\
\\
One of the outstanding facts in oriented matroid theory is that oriented matroids always admit topological representations, as established by Folkman and Lawrence~\cite{FL78}.
Here, we explain a variant of this fact, which was originally formulated in terms of {\it allowable sequences} by Goodman and Pollack~\cite{GP84}.
First, we observe that oriented matroids also arise from generalization of point configurations, called {\it pseudoconfigurations of points} 
(also called {\it generalized configurations of points}). 
\begin{defn}(Pseudoconfigurations of points) \\
A pair $PP=(P,L)$ of a point configuration $P:=(p_e)_{e \in [n]}$ in $\mathbb{R}^2$ and a collection $L$ of unbounded Jordan curves
is called a {\it pseudoconfiguration of points} (or a {\it generalized configuration of points}) if the following hold.
\begin{itemize}
\item For any $l \in L$, there exist at least two points of $P$ lying on $l$.
\item For any two points of $P$, there exists a unique curve in $L$ that contains both points.
\item Any pair of (distinct) curves $l_1,l_2 \in L$ intersects at most once. 
\end{itemize}
\end{defn}
For each $l \in L$, we label the two connected components of $\mathbb{R}^2 \setminus l$  arbitrarily as $l^+$ and $l^-$.
Then, we assign the sign vector $X_{l} \in \{ +1,-1,0\}^n$ 
such that $X_{l}^0 =  \{ e \in [n] \mid  p_e \in l\}$, $X_{l}^+ = \{ e \in [n] \mid  p_e \in l^+\}$ and  
$X_{l}^- =\{ e \in [n] \mid  p_e \in l^-\}$,
and we let ${\cal C}^*_{PP} := \{ \pm X_{l} \mid l \in L\}$.
Then, ${\cal M}_{PP}=([n], {\cal C}^*_{PP})$ turns out to be an acyclic oriented matroid of rank $3$.
Goodman and Pollack~\cite{GP84} proved that in fact the converse also holds.
\begin{thm}{\rm (Topological representation theorem for acyclic oriented matroids of rank $3$~\cite{GP84})}\\
For any acyclic oriented matroid ${\cal M}$ of rank $3$, there exists a pseudoconfiguration of points $PP$ with ${\cal M} = {\cal M}_{PP}$.
\end{thm}
\begin{figure}[h]
\begin{center}
\includegraphics[bb= 0 0 344 314, scale=0.30]{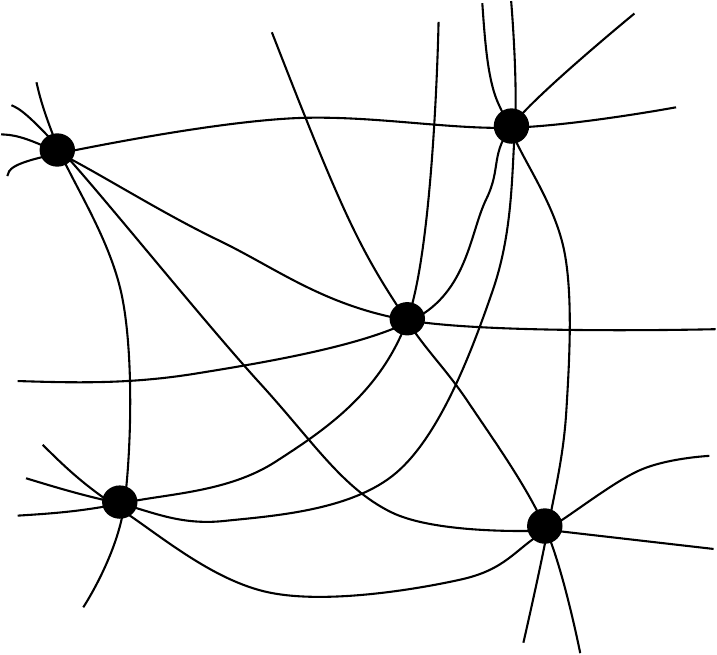}
\end{center}
\caption{ A pseudoconfiguration of points in $\mathbb{R}^2$}
\label{fig:ppc}
\end{figure}
Here, the assumption that ${\cal M}$ is acyclic is not important, because non-acyclic oriented matroids can be represented as {\it signed pseudoconfigurations of points},
where each point has a sign.

The notion of pseudoconfigurations of points in the $d(\geq 3)$-dimensional Euclidean space can be introduced analogously, but not every acyclic oriented matroid of rank $d+1$
can be represented as a pseudoconfiguration of points in the $d$-dimensional Euclidean space. For further details, see \cite[Section 5.3]{BLSWZ99}.
However, oriented matroids of general rank can be represented as pseudosphere arrangements~\cite{FL78}, with further details presented in \cite[Section 5.2]{BLSWZ99}.

\section{Definition of degree-$k$ oriented matroids}
In this section, we introduce degree-$k$ oriented matroids as a combinatorial model, which captures the essential combinatorial properties of configurations of points and the graphs of polynomial functions of degree $k$.
To do so, we first review some results that were introduced by Eli\'a\v{s} and Matou\v{s}ek~\cite{EM13}.

Eli\'a\v{s} and Matou\v{s}ek~\cite{EM13} introduced the notion of $k$th-order monotonicity, which is a generalization of the usual notion of monotonicity, described as follows.
(Because it is more convenient in our context  to reinterpret the $(k+1)$st-order monotonicity of   Eli\'a\v{s} and Matou\v{s}ek as the $k$th-order monotonicity, the following definitions are slightly different from the originals.)
Let $P= \{ p_1,\dots,p_n \}$ be a point set in the plane with $x(p_1)  < \cdots < x(p_n)$.
We assume that $P$ is in {\it $k$-general position}, i.e., no $k+2$ points of $P$ lie on the graph of a polynomial function of degree at most $k$.
Furthermore, we define a $(k+2)$-tuple of points in $P$ to be {\it positive} (resp. {\it negative}) if they lie on the graph of a function whose ($k+1$)st order derivative is everywhere nonnegative (resp. nonpositive).
A subset $S \subset [n]$ is said to be {\it $k$th-order monotone} if its $(k+2)$-tuples are either all positive or all negative. 
This notion can be stated in an alternative manner using the map $\chi_P^k : [n]^{k+2} \rightarrow \{ +1,-1,0\}$, defined as follows. 
\[ \chi_P^k (i_1,\dots,i_{k+2}) := {\rm sign}(N_{i_1,\dots,i_{k+1}}(p_{i_{k+2}})),\]
where $N_{i_1,\dots,i_{k+1}}$ is the Newton interpolation polynomial of the points $p_{i_1},\dots,p_{i_{k+1}}$, i.e., 
$y = N_{i_1,\dots,i_{k+1}}(x)$ is the unique polynomial function of  degree $k$ whose graph passes through the points $p_{i_1},\dots,p_{i_{k+1}}$.
Under this definition, a subset $S$ is $k$th-order monotone if and only if 
$\chi_P^k (s_1,\dots,s_{k+2}) = +1$ for all $(s_1,\dots,s_{k+2}) \in \Lambda (S,k+2)$ (see \cite[Lemma 2.4]{EM13}).
This map contains information on which side of the graph of the polynomial  function of degree $k$ determined by $p_{i_1},\dots,p_{i_{k+1}}$ the point $p_{i_{k+2}}$ lies.
In other words, the map $\chi_P^k$ contains information regarding the partitions of $P$ by graphs of polynomial functions of degree $k$.
\\
\\
It is shown in \cite{EM13} that the map $\chi_P^k$ can be computed in terms of determinants of a higher-dimensional space.
\begin{prop}{\rm ($(k+2)$-dimensional linear representability~\cite[Lemma 5.1]{EM13})}\\
Let $f : \mathbb{R}^2 \rightarrow \mathbb{R}^{k+2}$ be the map that sends each point $(x,y) \in \mathbb{R}^2$ to $(1,x,\dots,x^k,y) \in \mathbb{R}^{k+2}$.
Then, we have
\begin{align*}
\chi_P^k (i_1,\dots,i_{k+2}) = {\rm sign}(\det (f(p_{i_1}),\dots,f(p_{i_{k+2}}))) \text{ for all $(i_1,\dots,i_{k+2}) \in [n]^{k+2}$},
\end{align*}
i.e., the map $\chi_P^k$ is a chirotope of an oriented matroid of rank $k+2$.
\label{prop:lifting}
\end{prop}
\begin{proof}
The proof is not difficult, and we refer the reader to \cite{EM13}. The latter statement will be proved in Proposition~\ref{prop:pp_def_om} in a more general setting.
\end{proof}
\\
\\
Here, one can observe that the map $\chi_P^k$ actually arises from a single element lifting of a cyclic hyperplane arrangement in $\mathbb{R}^k$. We will see a generalization of this fact later (see Section~\ref{sec:conc}).

Eli\'a\v{s} and Matou\v{s}ek~\cite{EM13} additionally observed the following useful property.
\begin{prop}{\rm (Transitivity~\cite[Lemma 2.5]{EM13})}\\
If $\chi_{P}^k(I_{k+3} \setminus \{ i_{k+3} \}) = \chi_{P}^k(I_{k+3} \setminus \{ i_1 \})$ for $I_{k+3}:=(i_1,\dots,i_{k+3}) \in \Lambda ([n],k+3)$,
then we have $\chi_{P}^k(I) = \chi_{P}^k(I_{k+3} \setminus \{ i_{k+3} \})$ for all $I \in \Lambda(I_{k+3},k+2)$. 
\label{prop:tans}
\end{prop}
\begin{proof}
This proposition is proved in \cite{EM13}, using Newton interpolation polynomials.
Here, we provide a geometric proof for $k=2$. The generalization of this is straightforward.

Without loss of generality, we assume that $\chi^2_P(i_1,i_2,i_3,i_4) = +1$, which means that the point $p_{i_4}$ is above the graph $y=N_{i_1,i_2,i_3}(x)$.
Note that the graphs  $y=N_{i_1,i_2,i_3}(x)$ and $y=N_{i_2,i_3,i_4}(x)$ intersect at $p_{i_2}$ and $p_{i_3}$, and that they do not intersect elsewhere.
This indicates that the graph $y=N_{i_2,i_3,i_4}(x)$ lies above the graph $y=N_{i_1,i_2,i_3}(x)$ for $x > x(p_{i_3})$.
The point $p_{i_5}$ is above the graph $y=N_{i_2,i_3,i_4}(x)$ by the assumption $\chi^2_P(i_2,i_3,i_4,i_5)=+1$,  and it follows that $p_{i_5}$ lies above 
the graph $y=N_{i_1,i_2,i_3}(x)$, which implies that $\chi^2_P(i_1,i_2,i_3,i_5)=+1$.
The same argument can be applied when either of the graphs $y=N_{i_1,i_2,i_4}(x)$ or  $y=N_{i_1,i_3,i_4}(x)$ are considered instead of $y=N_{i_1,i_2,i_3}(x)$.
\begin{figure}[h]
\begin{center}
\includegraphics[scale=0.15, bb=0 0 947 816]{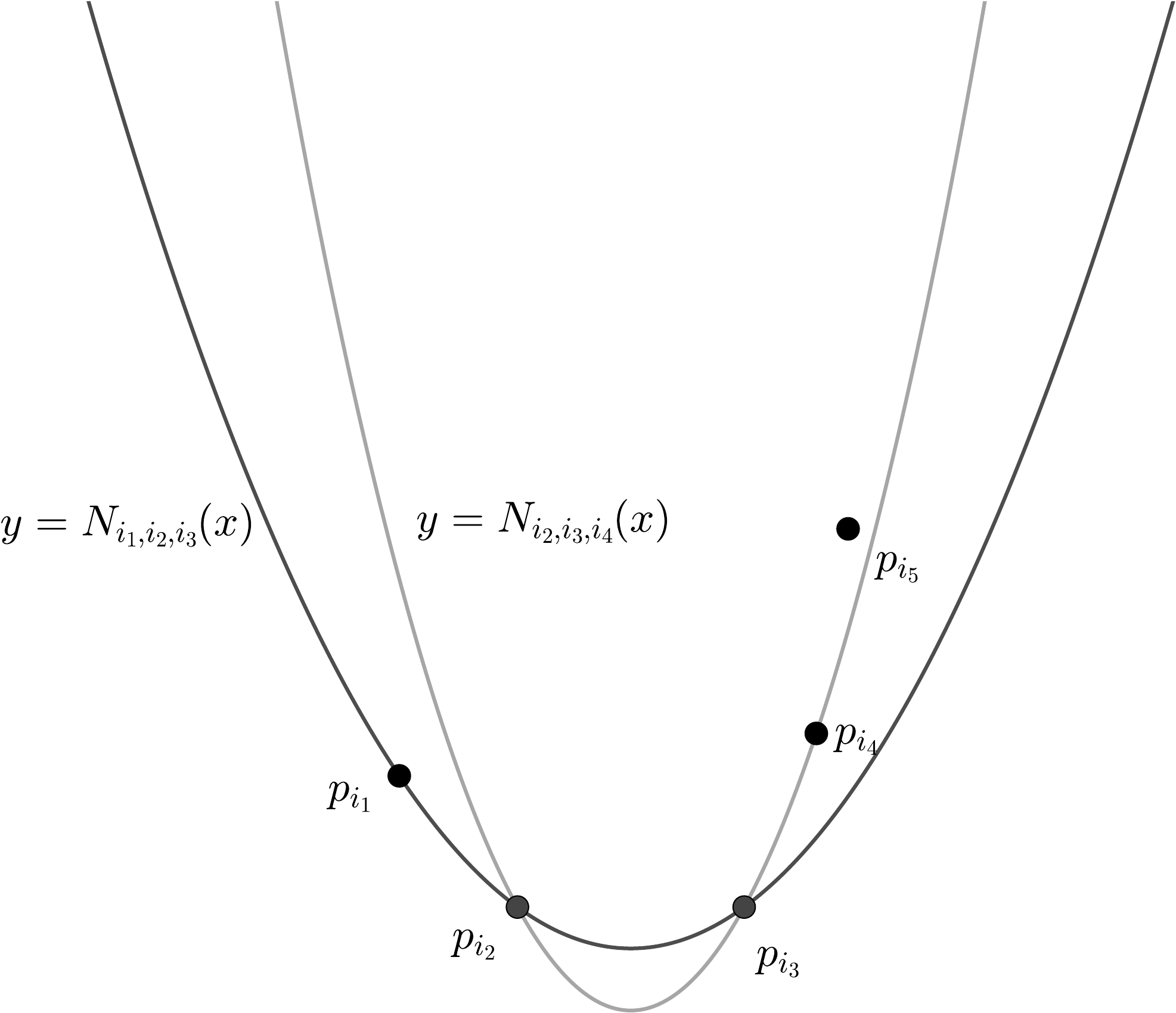}
\end{center}
\caption{ Configuration with $\chi^2_P(i_1,i_2,i_3,i_4) = +1$ and $\chi^2_P(i_2,i_3,i_4,i_5)=+1$}
\label{fig:trans}
\end{figure}
\end{proof}
\\
\\
In \cite{EM13}, this property is used to prove a new generalization of the Erd\H{o}s-Szekeres theorem, which states that
 there is always a $k$-monotone subset of size $\Omega(n)$ in any point set $P$ of size ${\rm twr}_k(n)$ in $k$-general position,
where ${\rm twr}_k(x )$ is the $k$th tower function.
Our first observation is that $\chi_{P}^k$ actually admits the following slightly stronger property.
\begin{prop}{\rm (($k+3$)-locally unimodal property)}\\
For any $\lambda \in \Lambda ([n],k+3)$ and $\mu_1,\mu_2,\mu_3 \in \Lambda (\bar{\lambda}, k+2)$ with $\mu_1 < \mu_2 < \mu_3$ (lexicographic order, i.e.,
there exist $a,b,c \in \bar{\lambda}$ ($a < b < c$) with $\mu_1 = \lambda \setminus \{ c \}, \mu_2 = \lambda \setminus \{ b \}, \mu_3 = \lambda \setminus \{ a \}$),
it holds that if $\chi_P^k (\mu_1) = - \chi_P^k (\mu_2)$, then $\chi_P^k (\mu_3) = \chi_P^k (\mu_2)$.
\end{prop}
\begin{proof}
The case with $k=2$ can easily be proved by careful examination of the proof of Proposition~\ref{prop:tans}.
We will prove this proposition later in a more general context (Proposition~\ref{prop:pp_def_om}), and so we omit the full proof here.
\end{proof}
\\
\\
In the next section, we will prove that the above-mentioned properties are in some sense {\it all} combinatorial properties that can be proved using some natural geometric properties.
This motivates us to consider combinatorial structures characterized by those properties.
\begin{defn}{\rm (Degree-$k$ oriented matroids)}\\
Let $E$ be a finite ordered set.
We say that a rank $k+2$ uniform oriented matroid ${\cal M} = (E, \{ \chi, -\chi \})$ is called a {\it degree-$k$  uniform oriented matroid} if it satisfies the following condition.  
For any $\lambda \in \Lambda (E,k+3)$ and $\mu_1,\mu_2,\mu_3 \in \Lambda (\bar{\lambda}, k+2)$ with $\mu_1 < \mu_2 < \mu_3$ (lexicographic order),
it holds that if $\chi (\mu_1) = - \chi (\mu_2)$, then $\chi (\mu_3) = \chi (\mu_2)$.
\label{def:degree-k}
\end{defn}
One can generalize  the definition of degree-$k$ oriented matroids to the non-uniform case as follows:
a degree-$k$ oriented matroid is a rank $k+2$ oriented matroid ${\cal M} = (E, \{ \chi, -\chi \})$ that satisfies the condition in Definition~\ref{def:degree-k} and in addition satisfies the condition `if $\chi (\mu_2) = 0$, then $\chi (\mu_1) = -\chi (\mu_3)$''. 
For the sake of simplicity, we only consider degree-$k$ uniform oriented matroids in the remaining part of the paper.

\section{Topological representation theorem for degree-$k$ oriented matroids}
In this section, we prove that degree-$k$ oriented matroids can always be represented by the following generalization of 
configurations formed by points and the graphs of polynomial functions of degree $k$.
\begin{defn}($k$-intersecting pseudoconfiguration of points)\\
\label{def:pp_configuration}
A pair $PP=(P,L)$ of a point set $P=\{ p_1,\dots,p_n \}$ ($x(p_1) < \cdots < x(p_n)$) in the plane and a collection $L$ of $x$-monotone Jordan curves is called a
{\it $k$-intersecting  pseudoconfiguration of points}  if the following conditions hold:
\begin{list}{}{}
\item[(PP1)] For any $l \in L$, there exist at least $k+1$ points of $P$ lying on $l$.
\item[(PP2)] For any $k+1$ points of $P$, there exists a unique curve $l \in L$ passing though each point.
\item[(PP3)] For any $l_1,l_2 \in L$ ($l_1 \neq l_2$), $l_1$ and $l_2$ intersect (transversally) at most $k$ times. 
\end{list} 
\end{defn}
Here, a Jordan curve is called {\it $x$-monotone} if it intersects with any vertical line at most once.
For $l \in L$, we denote $P(l) := P \cap l$.
If $|P(l)| = k+1$ for all $l \in L$, then the configuration $PP$ is said to be {\it in general position}.
When $PP$ is a configuration  in general position, we denote  the curve determined by points $p_{i_1},\dots,p_{i_{k+1}}$ by $l_{i_1,\dots,i_{k+1}}$.
An $x$-monotone Jordan curve $l$ can be written as $l = \{ (x,f(x)) \mid x \in \mathbb{R}\}$, for some continuous function $f: \mathbb{R} \rightarrow \mathbb{R}$.
We define $l^+ := \{ (x,y) \mid y > f(x) \}$, $l^- := \{ (x,y) \mid y < f(x) \}$, and $\overline{l^{+}} := l^+ \cup l$, $\overline{l^{-}} := l^- \cup l$.
We call $l^+$ (resp. $l^-$) the {\it $(+1)$-side} (resp. {\it $(-1)$-side}) of $l$.
For $Q \subset P$, the {\it subconfiguration} induced by $Q$ is a $k$-intersecting pseudoconfiguration $(Q,L|_Q)$,
where $L|_Q := \{ l \in L \mid P(l) \subset  Q\}$.
For $I \subset \mathbb{R}$, we denote $H_I := I \times \mathbb{R} \subset \mathbb{R}^2$.
\\
\begin{figure}[h]
\begin{center}
\includegraphics[bb= 0 0 405 313, scale=0.30]{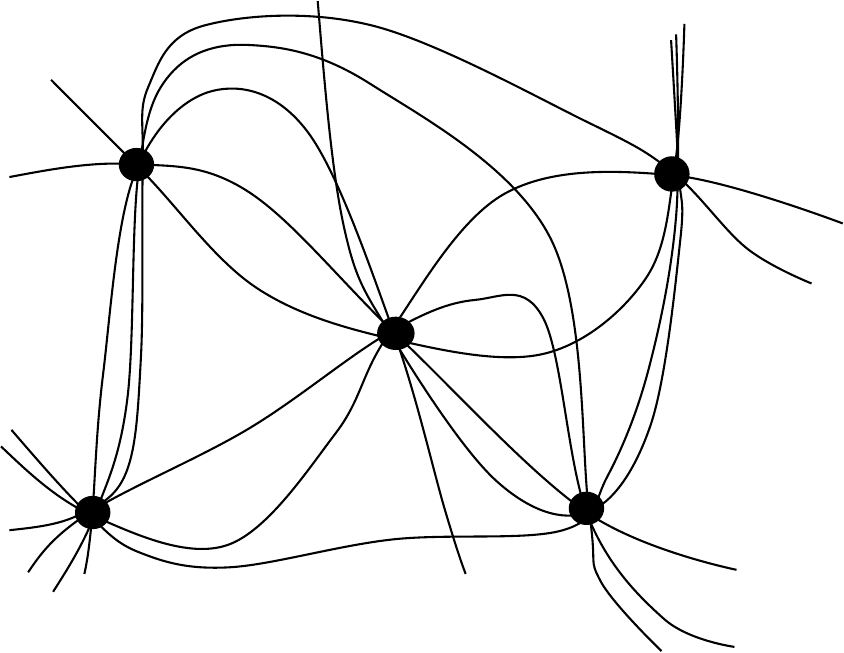}
\end{center}
\caption{ A $2$-intersecting pseudoconfiguration of points}
\label{fig:kpl}
\end{figure}
\\
For each $k$-intersecting pseudoconfigurations of points,  the possible partitions can be described as a function similarly to the case of configurations formed by points and the graphs of polynomial functions of degree $k$.  
To define and analyze this, we require the following notion.
\begin{defn}(Lenses)\\
Let $(P,L)$ be a $k$-intersecting pseudoconfiguration of points.
A {\it lens} (or {\it full lens}) of $(P,L)$ is a region that can be represented as $\overline{l_1^+} \cap \overline{l_2^-} \cap H_{[x_1,x_2]}$, where $l_1, l_2 \in L$, and $x_1$ and $x_2$ are the $x$-coordinates of two consecutive
intersection points of $l_1$ and $l_2$.
Note that $l_1 \cap l_2 \cap H_{[x_1,x_2]}$ consists of two points. We call these the {\it end points} of $R$. The end point with the greater (resp. smaller) 
$x$-coordinate is called the right (resp. left) end point, and is denoted by $e_r(R)$ (resp. $e_l(R)$).

A {\it half lens} of $(P,L)$ is a region represented as $\overline{l_1^+} \cap \overline{l_2^-} \cap H_{[x_1,\infty)}$ 
or $\overline{l_1^+} \cap \overline{l_2^-} \cap H_{(-\infty,x_2]}$, where $l_1, l_2 \in L$, and $x_1$ and $x_2$ are the $x$-coordinates of 
the leftmost and rightmost intersection points of $l_1$ and $l_2$, respectively.
We call a full or half lens $R$ an {\it empty lens} if $P \cap R = \emptyset$.
\end{defn}

Given a $k$-intersecting pseudoconfiguration of points $PP=(P,L)$  in general position (or more generally a configuration satisfying only (PP1) and (PP2)), 
we define a map $\chi_{PP} : [n]^{k+2} \rightarrow \{ +1,-1,0\}$ as follows.
\begin{itemize}
\item For $\lambda_1,\dots,\lambda_{k+2} \in [n]$ with $\lambda_1 < \cdots < \lambda_{k+1}$, we have  
\begin{align*}
\chi_{PP} (\lambda_1, \dots, \lambda_{k+2}) =
\begin{cases}
+1 & \text{ if $p_{\lambda_{k+2}} \in l_{\lambda_{1},\dots,\lambda_{k+1}}^+$,} \\
-1 & \text{ if $p_{\lambda_{k+2}} \in l_{\lambda_{1},\dots,\lambda_{k+1}}^-$.} 
\end{cases}
\end{align*}
\item $\chi_{PP} (\lambda_1,\dots,\lambda_{k+2}) = 0$ if $\lambda_i = \lambda_j$ for some $i, j \in [k+2]$ ($i \neq j$).
\item $\chi_{PP} (\lambda_{\sigma (1)}, \dots, \lambda_{\sigma(k+2)}) = {\rm sgn}(\sigma )\chi_{PP} (\lambda_1,\dots,\lambda_{k+2})$ for any $\lambda_1,\dots,\lambda_{k+2} \in [n]$ and any permutation $\sigma$ on $[k+2]$.
\end{itemize}

\begin{prop}
The map $\chi_{PP}$ is well-defined.
\end{prop}
\begin{proof}
It suffices to show that for any permutation $\sigma$ on $[k+2]$ with $\sigma (1) < \cdots < \sigma (k+1)$ and any sign $s \in \{ +,-\}$, 
we have $p_{\lambda_{\sigma(k+2)}} \in l_{\lambda_{\sigma(1)},\dots,\lambda_{\sigma(k+1)}}^{{\rm sgn} (\sigma) \cdot s} \Leftrightarrow p_{\lambda_{k+2}} \in l_{\lambda_{1},\dots,\lambda_{k+1}}^{s}$.
First, we pick such a $\sigma$ arbitrarily. Then, we have 
\begin{align*}
\sigma (i) = 
\begin{cases}
i & \text{ if $i \leq i_0$,}\\
i+1& \text{ if $i_0+1 \leq i \leq k+1$,}\\
i_0+1 & \text{ if $i = k+2$}
\end{cases}
\end{align*}
for some $i_0 \in [k+1]$ and ${\rm sgn}(\sigma) = (-1)^{k+1-i_0}$. 
If $i_0 = k+1$, then the proposition is trivial, and thus we assume that $i_0 \neq k+1$.
Then, the curves $l_{\lambda_{1},\dots,\lambda_{k+1}}$ and $l_{\lambda_{\sigma(1)},\dots,\lambda_{\sigma(k+1)}}$ intersect 
at the points $p_{\lambda_1},\dots,p_{\lambda_{i_0}},p_{\lambda_{i_0+2}},\dots,p_{\lambda_{k+1}}$. 
By the condition of $k$-intersecting pseudoconfigurations of points, these two curves must not intersect elsewhere,
and thus the curve $l_{\lambda_{\sigma(1)},\dots,\lambda_{\sigma(k+1)}}$ must lie above $l_{\lambda_{1},\dots,\lambda_{k+1}}$ in the halfspace $H_{(x(p_{\lambda_{k+1}}), \infty)}$
if $p_{\lambda_{k+2}}$ is above $l_{\lambda_{1},\dots,\lambda_{k+1}}$.
Because the above-below relationship of $l_{\lambda_{1},\dots,\lambda_{k+1}}$ and $l_{\lambda_{\sigma(1)},\dots,\lambda_{\sigma(k+1)}}$ is reversed at each end point of each lens formed by these two curves, 
the curve $l_{\lambda_{\sigma(1)},\dots,\lambda_{\sigma(k+1)}}$ must lie above (resp. below) $l_{\lambda_{1},\dots,\lambda_{k+1}}$
in the space  $H_{(x(p_{\lambda_{i_0-1}}), x(p_{\lambda_{i_0+1}}))}$  if $k-i_0$ is even, i.e., if ${\rm sgn}(\sigma) = +1$ (resp. if $k-i_0$ is odd, i.e., if ${\rm sgn}(\sigma) = -1$).
Therefore, we have $p_{\sigma (k+2)} = p_{\lambda_{i_0}} \in l_{\lambda_{\sigma(1)},\dots,\lambda_{\sigma(k+1)}}^{{\rm sgn}(\sigma)}$.
The same discussion applies in the case that $p_{\lambda_{k+2}}$ is below $l_{\lambda_{1},\dots,\lambda_{k+1}}$.
\begin{figure}[h]
\begin{center}
\includegraphics[scale=0.30, bb = 0 0 724 274]{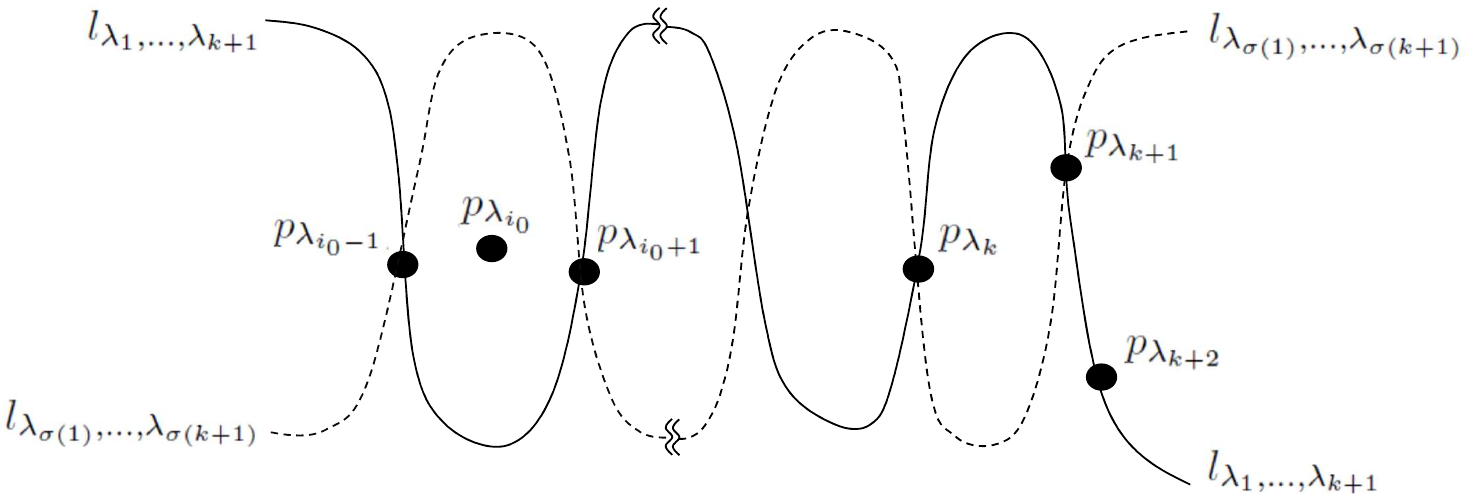}
\end{center}
\caption{$l_{\lambda_{1},\dots,\lambda_{k+1}}$ and $l_{\lambda_{\sigma(1)},\dots,\lambda_{\sigma(k+1)}}$}
\label{fig:well_def}
\end{figure}
\end{proof}

\subsection*{Step 1: $k$-intersecting pseudoconfigurations of points determine degree-$k$ oriented matroids}

\begin{prop}
For every $k$-intersecting pseudoconfiguration of points $PP=(P=\{ p_e\}_{e\in [n]},L)$  in general position, the map $\chi_{PP}$ defines a degree-$k$ uniform oriented matroid.
\label{prop:pp_def_om}
\end{prop}
\begin{proof}
First, we prove that the map $\chi_{PP}$ is a chirotope of an oriented matroid of rank $k+2$.
For this, it suffices to prove that the set ${\cal C}^*_{PP} := \{  \pm (\chi_{PP} (\lambda, 1), \dots, \chi_{PP} (\lambda, n)) \mid  \lambda \in [n]^{k+1} \} \setminus \{ 0 \}$ fulfills the cocircuit axioms of oriented matroids.
Clearly, Axioms (C0)--(C2) are satisfied and we need only verify Axiom (C3).
Because we have $|X^0| = k+1$ for all $X \in {\cal C}^*_{PP}$, it suffices to verify Axiom (C3').
Take $\lambda, \mu \in \Lambda ([n],k+1)$ with $|\bar{\lambda} \cap \bar{\mu} | = k$.
Let $X,Y \in {\cal C}^*_{PP}$  be sign vectors that correspond to $l_{\lambda}$ and $l_{\mu}$, respectively ($X$ and $Y$ are determined uniquely up to a sign reversal).
Take an $e \in (X^+ \cap Y^-) \cup (X^- \cap Y^+)$ and
any $f_0 \in  (X^+ \cap Y^+) \cup (X^- \cap Y^-)$, and let $Z \in {\cal C}^*_{PP}$ be the sign vector  with $Z(f_0) = X(f_0)$ that corresponds to $l_{\bar{\lambda} \cap \bar{\mu} \cup \{ e \}}$. 
We verify that $Z$ is a required cocircuit in Axiom (C3').
Note that $l_{\lambda}$ and $l_{\mu}$ form $k+1$ lenses (see Figure \ref{fig:pp_def_om}).
Because $l_{\bar{\lambda} \cap \bar{\mu} \cup \{ e \}}$ already intersects $k$ times with each of the curves $l_{\lambda}$ and $l_{\mu}$  at the points with indices in $\bar{\lambda} \cap \bar{\mu}$, 
it cannot intersect with $l_{\lambda}$ or $l_{\mu}$ elsewhere.
Therefore, if $p_e$ is contained inside of one of the lenses, then the whole of $l_{\bar{\lambda} \cap \bar{\mu} \cup \{ e \}}$ must lie in the lenses.
Take any $f \in [n]$ with $X(f) = Y(f) \neq 0$. Then, $p_f$ lies outside of the lenses formed by $l_{\lambda}$ and $l_{\mu}$,
and thus $p_f$ and $p_{f_0}$ lie on the same side of $l_{\bar{\lambda} \cap \bar{\mu} \cup \{ e \}}$.
If $p_e$ is outside of the lenses, then the whole of $l_{\bar{\lambda} \cap \bar{\mu} \cup \{ e \}}$ must lie outside of the lenses (except for the end points).
Points $p_f$ with $X(f) = Y(f) \neq 0$ corresponds to points inside of the lenses and a similar discussion shows that $X(f) = Z(f)$.
Therefore, we have $Z^+ \supset X^+ \cap Y^+$ and $Z^- \supset X^- \cap Y^-$. 

\begin{figure}[h]
\begin{center}
\includegraphics[scale=0.35, bb= 0 0 550 279]{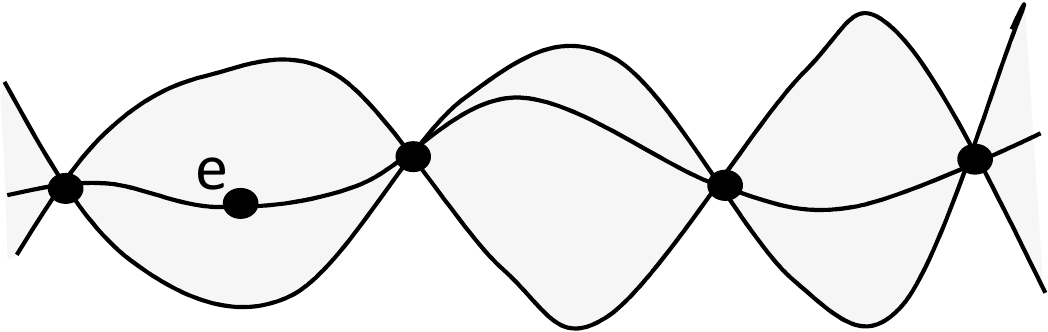}
\includegraphics[scale=0.35, bb= 0 0 550 279]{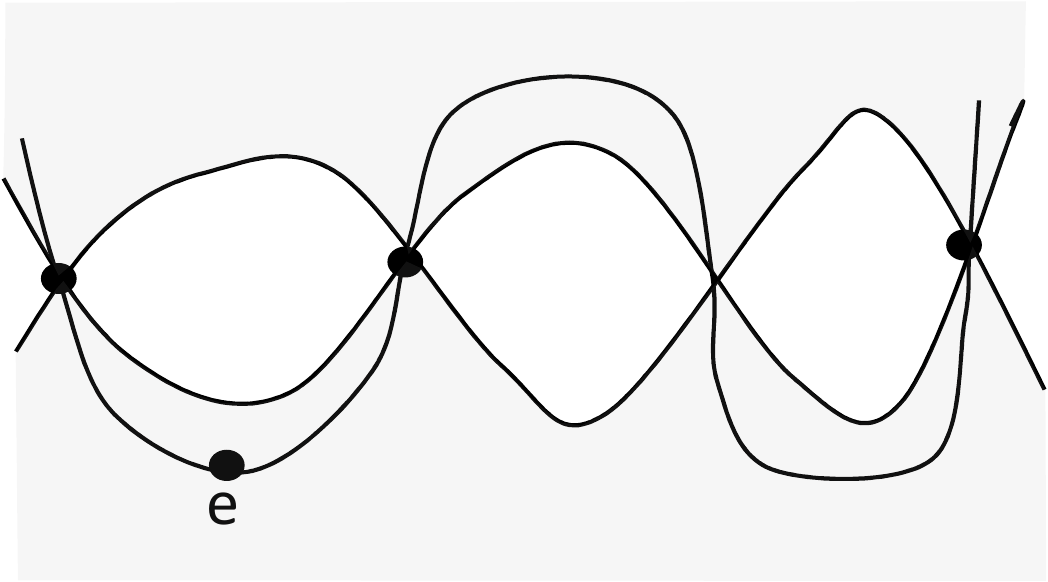}
\end{center}
\caption{ $l_{\lambda}$, $l_{\mu}$, and $l_{\bar{\lambda} \cap \bar{\mu} \cup \{ e \}}$}
\label{fig:pp_def_om}
\end{figure}

Next, we confirm the ($k+3$)-locally unimodal property.
Suppose that there exist $\lambda \in \Lambda ([n],k+3)$ and $\nu_1,\nu_2,\nu_3 \in \Lambda(\bar{\lambda}, k+2)$ such that $\nu_1 < \nu_2 < \nu_3$
and $\chi_{PP} (\nu_1) = \chi_{PP} (\nu_2) = +1$, $\chi_{PP} (\nu_3) = -1$.
Let $\mu = \nu_1 \cap \nu_2 \cap \nu_3 \in \Lambda (\bar{\lambda}, k)$ and $a,b,c$ ($a < b < c$) be the integers such that 
$\bar{\nu_1} = \bar{\mu} \cup \{ a,b \}$, $\bar{\nu_2} = \bar{\mu} \cup \{ a,c \}$, and $\bar{\nu_3} = \bar{\mu} \cup \{ b,c \}$.
Let $i_a$ be the integer such that $\mu_{i_a} < a < \mu_{i_a+1}$, where we assume that $\lambda_0 = -\infty$ and $\lambda_{k+4} = \infty$.
Define integers $i_b$ and $i_c$ similarly.
Since $\chi_{PP} (\nu_1) = \chi_{PP} (\nu_3) =  +1$,
the point $p_b$ is on the $(-1)^{k+1-i_b}$-side of $l_{\mu \cup \{ a\}}$ and on the $(-1)^{k+2-i_b}$-side of $l_{\mu \cup \{ c\}}$.
Because $l_{\mu \cup \{ a\}}$ and $l_{\mu \cup \{ c\}}$ must not intersect more than $k$ times, the curves $l_{\mu \cup \{ a\}}$ and $l_{\mu \cup \{ c\}}$ form lenses with
end points $p_{\mu_1}, \dots, p_{\mu_k}$,
and $l_{\mu \cup \{ b\}}$ must lie inside of these lenses.
Now, we remark that the point $p_c$ is on the $(-1)^{k+1-i_c}$-side of $l_{\mu \cup \{ a\}}$ and the $(-1)^{k+2-i_c}$-side of $l_{\mu \cup \{ b\}}$, because
$\chi_{PP} (\nu_2) = -1$ and $\chi_{PP} (\nu_3) = +1$.
Therefore, the curve $l_{\mu \cup \{ c\}}$ must intersect with either of the curves $l_{\mu \cup \{ a\}}$ and $l_{\mu \cup \{ b\}}$ in $H_{(x(p_{\mu_{i_c}}), x(p_c))}$.
This means that $l_{\mu \cup \{ c\}}$ must intersect at least $k+1$ times with either of the curves $l_{\mu \cup \{ a\}}$ and $l_{\mu \cup \{ b\}}$, which is a contradiction. 
\begin{figure}[h]
\begin{center}
\includegraphics[scale=0.4, bb =0 0 517 227]{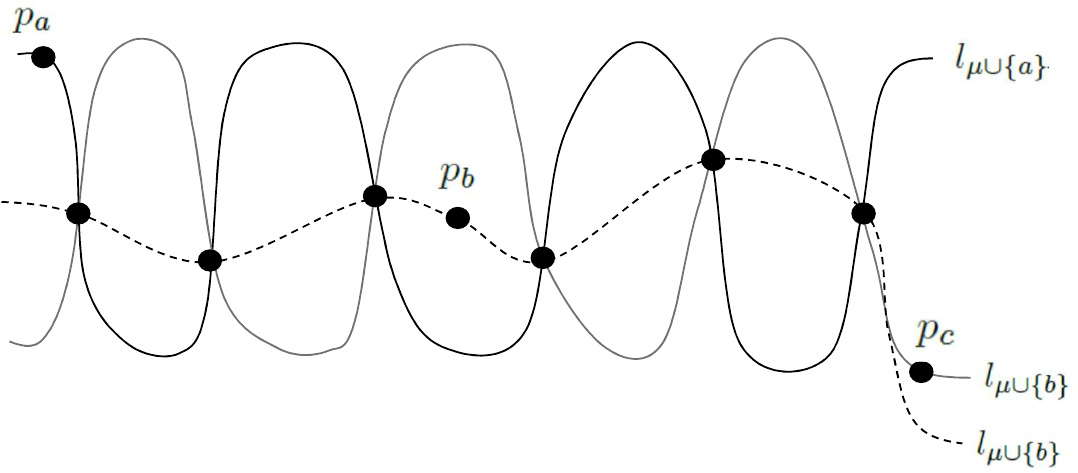}
\end{center}
\caption{$l_{\mu \cup \{ a\}}$, $l_{\mu \cup \{ b\}}$, and $l_{\mu \cup \{ c\}}$}
\end{figure}
\end{proof}

\subsection*{Step 2: Degree-$k$ oriented matroids can be represented as $k$-intersecting pseudoconfigurations of points}
Here, we prove that every degree-$k$ oriented matroid admits a topological representation as a $k$-intersecting pseudoconfiguration of points.
To this end, we introduce two operations.
\begin{defn}(Empty lens elimination I)\\
Let $R$ be an empty lens represented as $R= \overline{l_1^+} \cap \overline{l_2^-} \cap H_{[x_1,x_2]}$, 
where $l_1,l_2 \in L$, and $x_1$ (resp. $x_2$) can be $-\infty$ (resp. $\infty$).
Transform $l_1$ and $l_2$ by connecting $l_1 \cap H_{(-\infty, x_1 -\epsilon]}$, $l_2 \cap H_{[x_1 + \epsilon, x_2 -\epsilon]}$, and $l_1 \cap H_{[x_2 +\epsilon, \infty )}$ (when $x_1 \neq -\infty$),
and by connecting $l_2 \cap H_{(-\infty, x_1 -\epsilon]}$, $l_1 \cap H_{[x_1 + \epsilon, x_2 -\epsilon]}$, and $l_2 \cap H_{[x_2 +\epsilon, \infty)}$  (when $x_2 \neq \infty$),
 for sufficiently small $\epsilon > 0$, so that the new curves do not have intersections around the vertical lines $x = x_1$ and $x= x_2$ (see Figure \ref{fig:transform1}). 
\end{defn}

\begin{figure}[h]
\begin{center}
\includegraphics[scale=0.25, bb = 40 380 528 772]{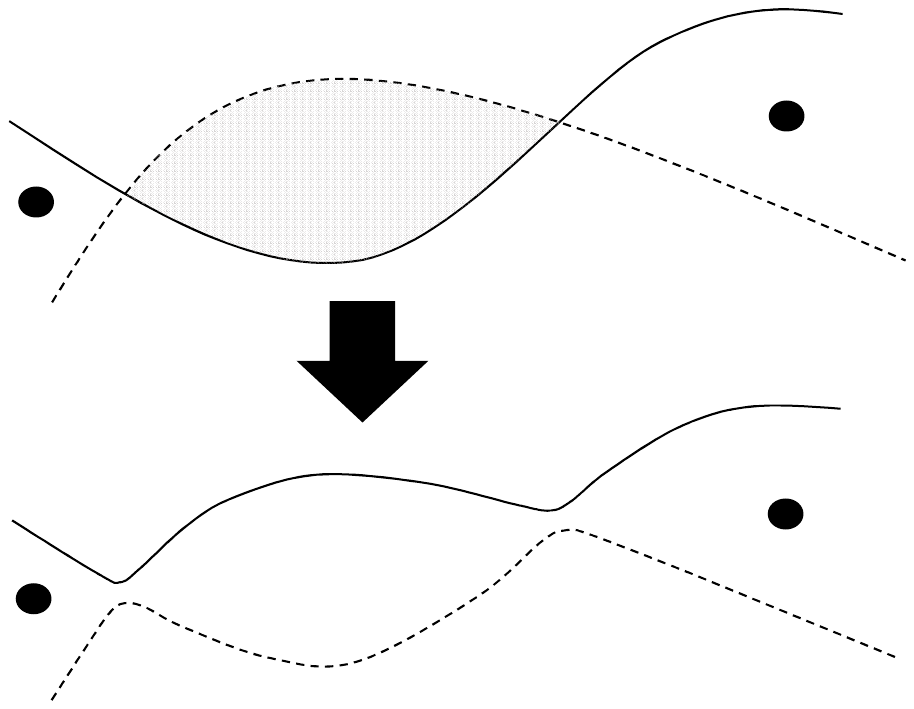}
\includegraphics[scale=0.25, bb = 40 380 528 772]{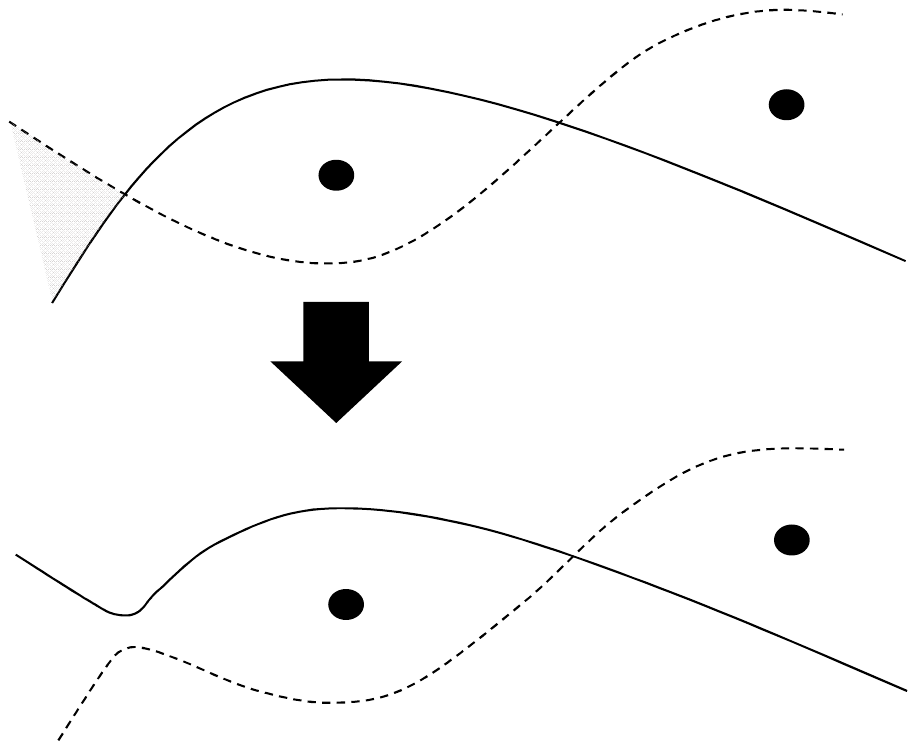}
\end{center}
\caption{Empty lens elimination I (empty lenses are shaded)}
\label{fig:transform1}
\end{figure}

\begin{defn}(Empty lens elimination II)\\
Let $R_1,\dots,R_n$ be lenses represented as $R_1= \overline{l_1^+} \cap \overline{l_2^-} \cap H_{[x_1,x_2]}$ and 
$R_2= \overline{l_1^-} \cap \overline{l_2^+} \cap H_{[x_2,x_3]}$, \dots,
$R_n= \overline{l_{n \bmod 2}^-} \cap \overline{l_{n \bmod 2 + 1}^+} \cap H_{[x_n,x_{n+1}]}$
where $l_1,l_2 \in L$, and $x_1=  -\infty$, $x_{n+1} = \infty$.
Suppose that $R_i \setminus \{ e_r(R_i)\}$, $R_{i+1} \setminus \{ e_l(R_{i+1}),e_r(R_{i+1})\}$, \dots,  $R_{j-1} \setminus \{ e_l(R_{j-1}),e_r(R_{j-1})\}$
and $R_j \setminus  \{ e_l(R_j)\}$ are empty for some $1 \leq i < j \leq n$.
Transform $l_1$ and $l_2$ by connecting $l_1 \cap H_{(-\infty, x_{i+1} -\epsilon]}$, $l_2 \cap H_{[x_{i+1} + \epsilon, x_j -\epsilon]}$, and $l_1 \cap H_{[x_j +\epsilon, \infty)}$,
and by connecting $l_2 \cap H_{(-\infty, x_{i+1} -\epsilon]}$, $l_1 \cap H_{[x_{i+1} + \epsilon, x_j -\epsilon]}$, and $l_2 \cap H_{[x_j +\epsilon, \infty)}$, for sufficiently small $\epsilon > 0$,
so that the new curves do not have intersections around the vertical lines $x = x_{i+1}$ and $x= x_j$
(see Figure \ref{fig:transform2}). 
\end{defn}

\begin{figure}[h]
\begin{center}
\includegraphics[scale=0.25, bb = 23 310 568 837]{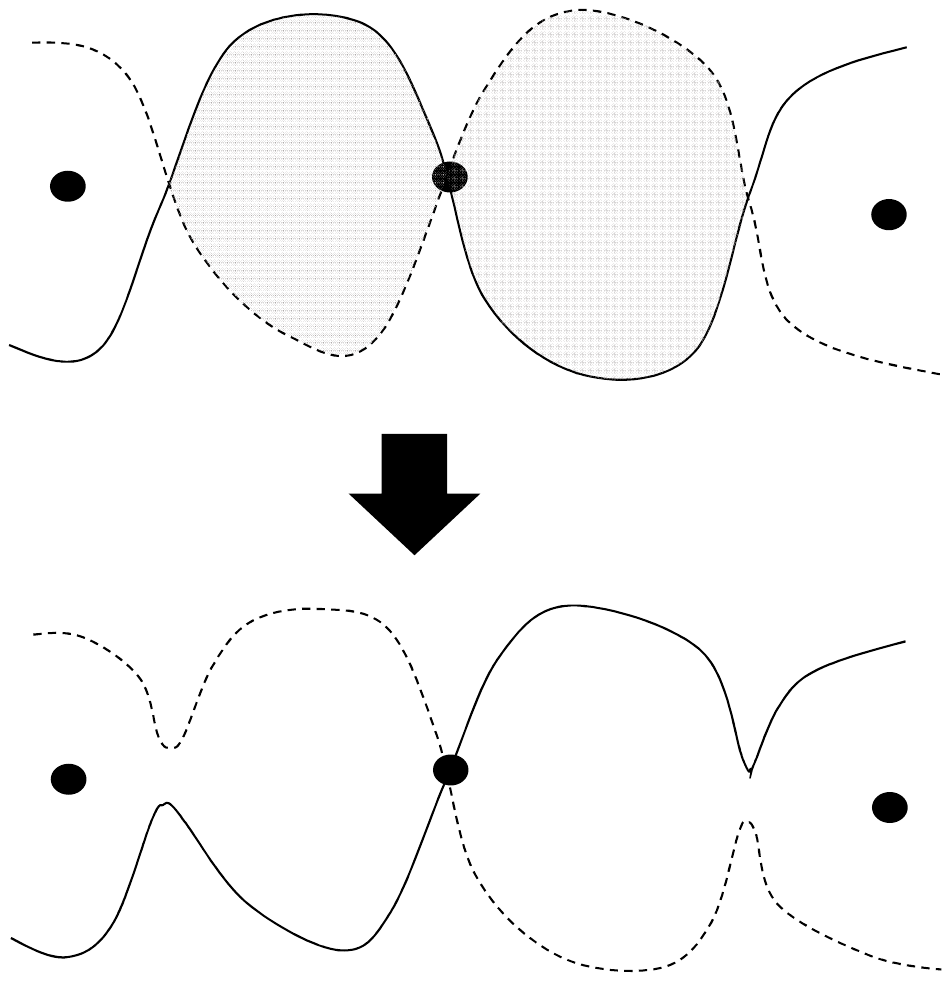} 
\includegraphics[scale=0.25, bb = 23 310 568 837]{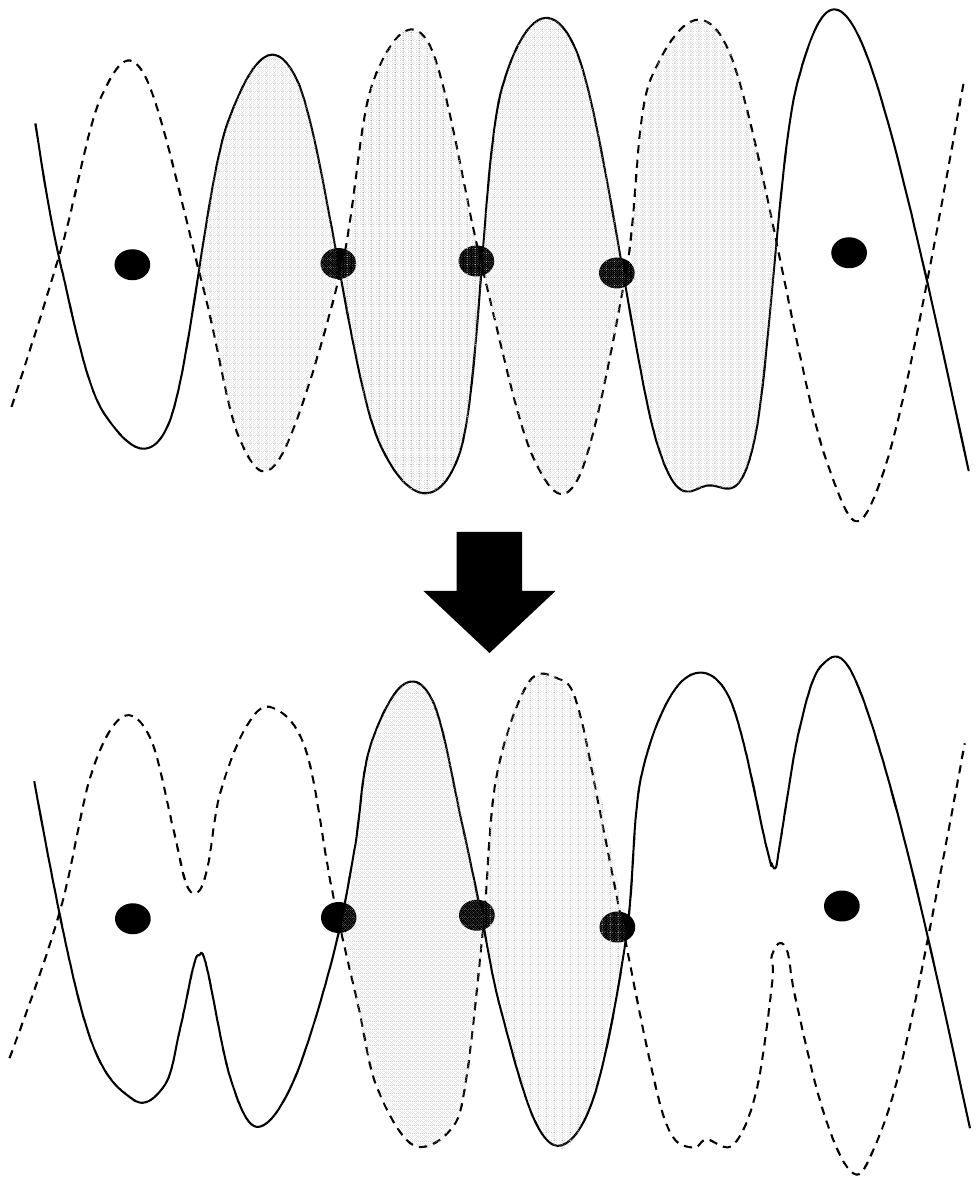} 
\end{center}
\caption{Empty lens elimination II (empty lenses are shaded)}
\label{fig:transform2}
\end{figure}

The two operations described above decrease the total number of intersection points of the curves in $L$, without altering the above-below relationships between the points in $P$ and the curves in $L$.
However, they may invalidate the condition that each pair of curves in $L$ intersect at most $k$ times.
However, we prove that a $k$-intersecting pseudoconfiguration of points is always obtained after the operations are applied as far as possible.

\begin{lem}
Let $PP=(P,L) $ be a configuration  in general position with a finite point set $P$ in $\mathbb{R}^2$ and a collection $L$ of $x$-monotone Jordan curves
satisfying (PP1) and (PP2) in Definition \ref{def:pp_configuration} (with some pair of curves possibly intersecting more than $k$ times).
We assume that the map $\chi_{PP}$ fulfills the axioms of degree-$k$ oriented matroids.
If it is impossible to apply the empty lens eliminations to $PP$, then $PP$  is a $k$-intersecting pseudoconfiguration of points.
\end{lem}
\begin{proof}
Recall the notation listed in the end of Section 1.
Assume that there is a pair of curves in $L$ that intersect more than $k$ times.
Let $\widetilde{P}$ be a minimal subset of $P$
such that there are two curves in $L|_{\widetilde{P}}$ intersecting more than $k$ times after the empty lens eliminations are applied as far as possible. 
Let $l_1$ and $l_2$ be curves in $L|_{\widetilde{P}}$ that have the smallest $x$-coordinate for the $(k+1)$st intersection point.
The curves $l_1$ and $l_2$ form $k_0 \geq k$ full lenses and two half lenses.
Let us label these $k_0+2$ full and half lenses by $R_1,\dots,R_{k_0+2}$ in increasing order with respect to the $x$-coordinates.
Because the empty lens eliminations I and II cannot be applied, there must exist $k_0+2$ distinct points $p_{i_1},\dots,p_{i_{k_0+2}}$ with
$p_{i_1} \in R_1, \dots, p_{i_{k_0+2}} \in R_{k_0+2}$. 
Note that there are no points outside of the union $R_1 \cup \cdots \cup R_{i_{k_0+2}}$ of the lenses, by the minimality of $\widetilde{P}$.
Let $I_{k+1} := (i_1,\dots,i_{k+1})$ and $I_{k+2} := (i_1,\dots,i_{k+2})$. 
We now consider several possible cases separately.
\\
\\
{\bf (Case I)} $l_{I_{k+1}} \neq l_1, l_2$.

The curve $l_{I_{k+1}}$ must intersect with $l_1$ and $l_2$ at least twice in total  on the boundary of each of the full lenses $R_2,\dots,R_{{k+1}}$ (to enter and to leave the lens, where passing through  an end point is counted as twice. If $l_{I_{k+1}}$ passes through $e_l(R_j)$ and $e_r(R_j)$ for some $j \in [k+1]$, we regard that $l_{I_{k+1}}$ intersects with $R_j$ at $e_l(R_j)$ and with $R_{j+1}$ at $e_r(R_j)$) 
and at least once on the boundary of each of the half or full lenses $R_1$ and $R_{k+2}$.
If $l_{I_{k+1}}$ intersects with  $l_1$ and $l_2$ more than twice in total on the boundary of some lens $R_a$ ($a \leq k+1$), this means that they intersect at least four times on the boundary of $R_a$, and that
either of the $(k+1)$st intersection point of $l_1$ and $l_{I_{k+1}}$ or that of $l_2$ and $l_{I_{k+1}}$ belongs to the halfspace $H_{(-\infty, x(e_r(R_{k+1})))}$. This contradicts the minimality assumption for the 
$x$-coordinate of the $(k+1)$st intersection point of $l_1$ and $l_2$.
Therefore, the curve $l_{I_{k+1}}$ actually intersects with $l_1$ and $l_2$ exactly twice in total on the boundary of each full lens and once or twice on the boundary of each half lens.
If $p_{i_1} \in R_1 \setminus \{ e_r(R_1) \}$, then
the curve $l_{I_k}$ intersects with $l_1$ and $l_2$ a total of $2k+1$ or $2k+2$ times in the halfspace $H_{(-\infty,x(e_r(R_k))]}$, and thus $k+1$ times with $l_1$ or $l_2$.
Without loss of generality, we assume that $l_{I_{k+1}}$ and $l_1$ have $k+1$ intersections in the closed halfspace $H_{(-\infty, x(e_r(R_{k+1}))]}$.
Because of the minimality assumption, the $(k+1)$st intersection point of $l_{I_{k+1}}$ and $l_1$ must coincide with $e_r(R_{k+1})$.
Let $R'_1,\dots,R'_{k+2}$ be the lenses formed by $l_1$ and $l_{I_{k+1}}$ (ordered by the $x$-coordinates).
We have $p_{i_1} \in l_{I_{k+1}} \cap R'_1$, \dots, $p_{i_{k+1}}(=e_r(R'_{k+1})) \in l_{I_{k+1}} \cap R'_{k+1}$.
On the other hand, if $p_{i_1} = e_r(R_1)$, then it must hold that $p_{i_{k+1}} \in R_{k+2} \setminus \{ e_r(R_{k+1}) \}$.
This mirrors the case with $p_{i_1} \in R_1 \setminus \{ e_r(R_1) \}$ and the remainder of the discussion proceeds in the same manner.
Therefore, we consider only the case with $p_{i_1} \in R_1 \setminus \{ e_r(R_1) \}$ in what follows.
We can classify the possible situations into the following two cases.
\\
\begin{figure}[h]
\begin{center}
\includegraphics[scale=0.4, bb = 0 0 805 131]{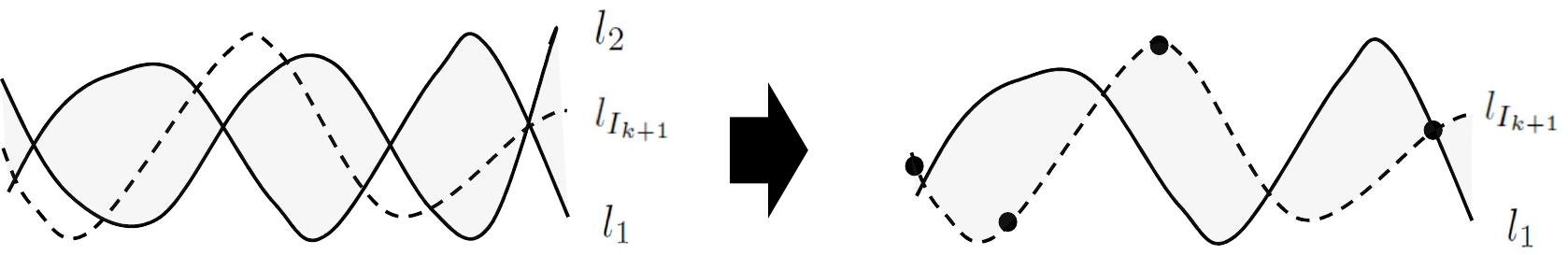} 
\end{center}
\caption{$l_1$, $l_2$, and $l_{I_{k+1}}$}
\label{fig:discussion1}
\end{figure}
\\
{\bf (Case I-i)} $(\widetilde{P} \setminus \widetilde{P}(l_{I_{k+1}})) \cap \bigcup_{i=1}^{k+1}{(R'_i \setminus \{ e_l(R'_i), e_r(R'_i) \})} \neq \emptyset$. 

Let $j := \min\{ u \in [n] \mid p_u \in (\widetilde{P} \setminus \widetilde{P}(l_{I_{k+1}})) \cap \bigcup_{i=1}^{k+1}{(R'_i \setminus \{ e_l(R'_i), e_r(R'_i) \})} \}$,
and let $j^* \in [k+1]$ be such that $p_{j} \in R'_{j^*} \setminus \{ e_l(R'_{j^*}), e_r(R'_{j^*}) \}$.
We first remark that the curves $l_{I_{k+1}[i_{j^*} | j]}$ and $l_{I_{k+1}}$ have $k+1$ intersections in the closed halfspace $H_{(-\infty,x(e_r(R'_{k+1}))]}$.
Otherwise, $l_{I_{k+1}[i_{j^*} | j]}$ and $l_1$ intersect $k+1$ times. Since $p_{i_{j^*}}$ is outside the lenses formed by $l_{I_{k+1}[i_{j^*} | j]}$ and $l_1$, this situation cannot be resolved by the empty lens eliminations in $PP|_{\widetilde{P} \setminus \{ p_{i_{j^*}}\} }$.
Since $\widetilde{P}(l_1) \cup \widetilde{P}(l_{I_{k+1}[i_{j^*} | j]}) \subset \widetilde{P} \setminus \{ p_{i_{j^*}} \}$, it contradicts the minimality assumption of $\widetilde{P}$.

Let $R^*_1,\dots,R^*_{k+2}$ be the 1st to the ($k+2$)nd lenses from the left formed by the curves $l_{I_{k+1}[i_{j^*} | j]}$ and $l_{I_{k+1}}$, which are labeled  in increasing order of the $x$-coordinates.
Then, we have 
$p_{i_l} = e_r(R^*_l)$ for all $l \in [k+1] \setminus \{ j^*\}$
and $p_j, p_{i_{j^*}} \in R^*_{j^*} \cup R^*_{j^*+1}$. 
Assume that $p_j \in R^*_{j^*}$ and  $p_{i_{j^*}} \in R^*_{j^*+1}$. Then, 
if $p_j$ is on the $\sigma$-side  ($\sigma \in \{ +1,-1\}$) of $l_{I_{k+1}}$, then $p_{i_{j^*}}$ must lie on the $\sigma$-side of $l_{I_{k+1}[i_{j^*}|j]}$ by the above-below relationship of $l_{I_{k+1}[i_{j^*} | j]}$ and $l_{I_{k+1}}$ (see Figures~{\ref{fig:discussion2} and \ref{fig:discussion2_2}}).
This implies that $\chi_{PP}(I_{k+1},j) = \chi_{PP}(I_{k+1}[i_{j^*} | j],i_{j^*})$, which contradicts Axiom (B2) of the chirotope axioms.
The case  that $p_j \in R^*_{j^*+1}$ and  $p_{i_{j^*}} \in R^*_{j^*}$ is also excluded by a similar discussion.
Therefore, we have $p_j, p_{i_{j^*}} \in R^*_{j^*}$ or $p_j, p_{i_{j^*}} \in R^*_{j^*+1}$.
Let us consider the former case assuming that $i_{j^*} < j$.
In this case, there must exist a point $p_{a} \in \widetilde{P} \cap (R^*_{a^*} \setminus \partial R^*_{a^*})$ for some $a^* \in [j^*+1, k+2]$ because otherwise empty lens elimination II can be applied
to $R^*_{j^*+1},\dots,R^*_{k+2}$.
Without loss of generality, we assume that $\chi_{PP}(I_{k+1},a) = +1$ and $\chi_{PP}(I_{k+1}[i_{j^*} | j],a) = -1$.
Then, we have
$\chi_{PP}(I_{k+1} [j,a] \setminus \{ j \}) = (-1)^{k-a^*}$ and
$\chi_{PP}(I_{k+1} [j,a] \setminus \{ i_{j^*} \}) = (-1)^{k-a^*+1}$,
where $I_{k+1}[j,a] := (i_1,\dots,i_{j^*},j,i_{j^*+1},\dots,i_{a^*-1},a,i_{a^*},\dots,i_{k+1})$.
By the $(k+3)$-locally unimodal property, we have $\chi_{PP}(I_{k+1}[j,a] \setminus \{ a \}) = (-1)^{k-a^*}$, and thus $\chi_{PP}(I_{k+1},j) = (-1)^{j^*-a^*+1}$.
This means that the point $p_j$ is on the $(-1)^{j^*-a^*+1}$-side of $l_{I_{k+1}}$, and thus that $l_{I_{k+1}[i_{j^*} | j]}$ is on the $(-1)^{j^*-a^*+1}$-side of $l_{I_{k+1}}$ at $x = x(p_j)$.
On the other hand, by the assumption $\chi_{PP}(I_{k+1},a) = +1$ and $\chi_{PP}(I_{k+1}[i_{j^*} | j],a) = -1$, the curve $l_{I_{k+1}[i_{j^*} | j]}$ must be above $l_{I_{k+1}}$ at $x = x(p_a)$.
Because the above-below relationship of $l_{I_{k+1}[i_{j^*} | j]}$ and $l_{I_{k+1}}$  is reversed at each end point of each lens,
the curve $l_{I_{k+1}[i_{j^*} | j]}$ must lie on the $(-1)^{a^*-j^*}$-side of $l_{I_{k+1}}$ at $x = x(p_j)$.
This is a contradiction.
A similar discussion also leads to a contradiction in the case that $i_{j^*} > j$.
Similarly, we obtain a contradiction in the case that $p_j, p_{i_{j^*}} \in R^*_{j^*+1}$.
\\
\begin{figure}[H]
\begin{center}
\includegraphics[scale=0.30, bb = 0 0 500 222]{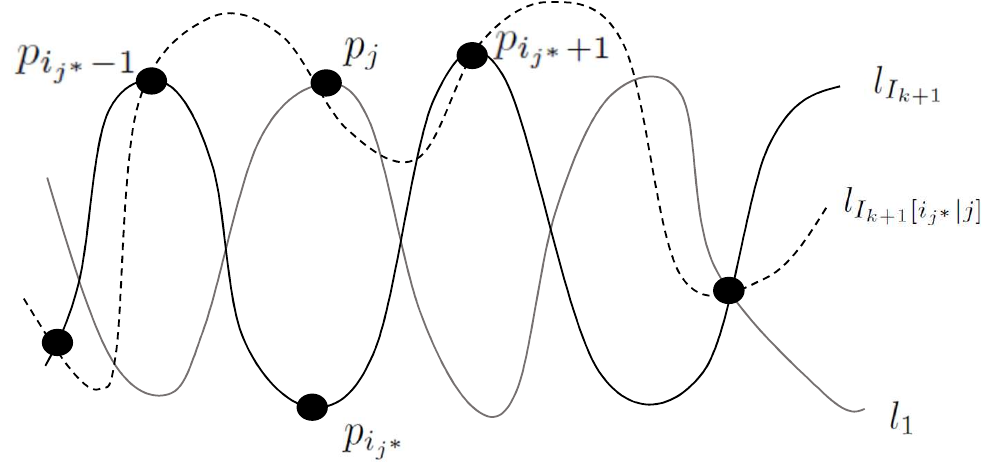} 
\includegraphics[scale=0.30, bb = 0 0 500 225]{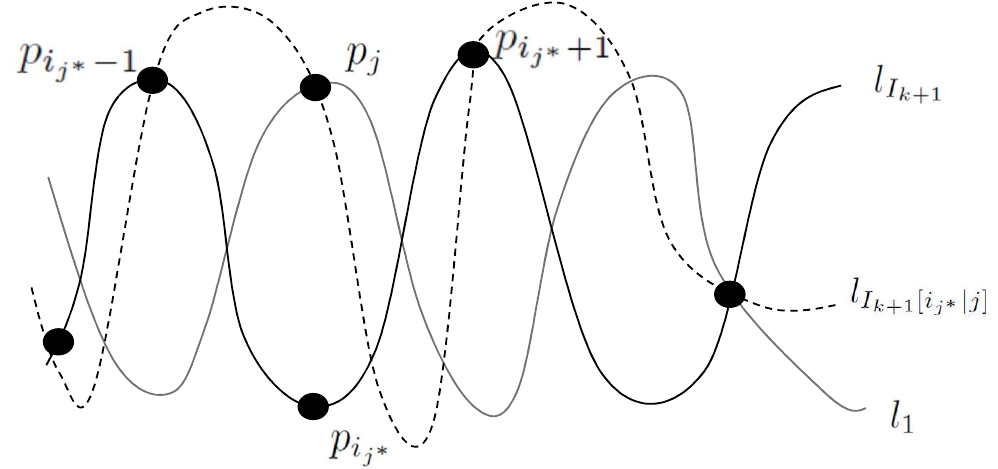}
\end{center}
\caption{$l_1$, $l_{I_{k+1}}$, and $l_{I_{k+1}[i_{j^*} | j]}$ ($p_j, p_{i_{j^*}} \in R^*_{j^*}$, $p_j \in \widetilde{P}(l_1)$)}
\label{fig:discussion2}
\end{figure}
\begin{figure}[H]
\begin{center}
\includegraphics[scale=0.30, bb = 0 0 500 237]{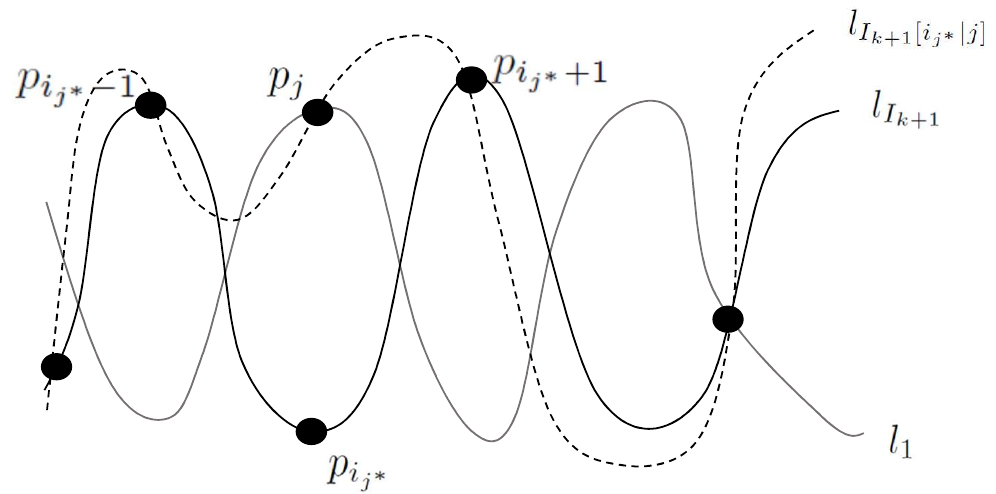} 
\includegraphics[scale=0.30, bb = 0 0 500 237]{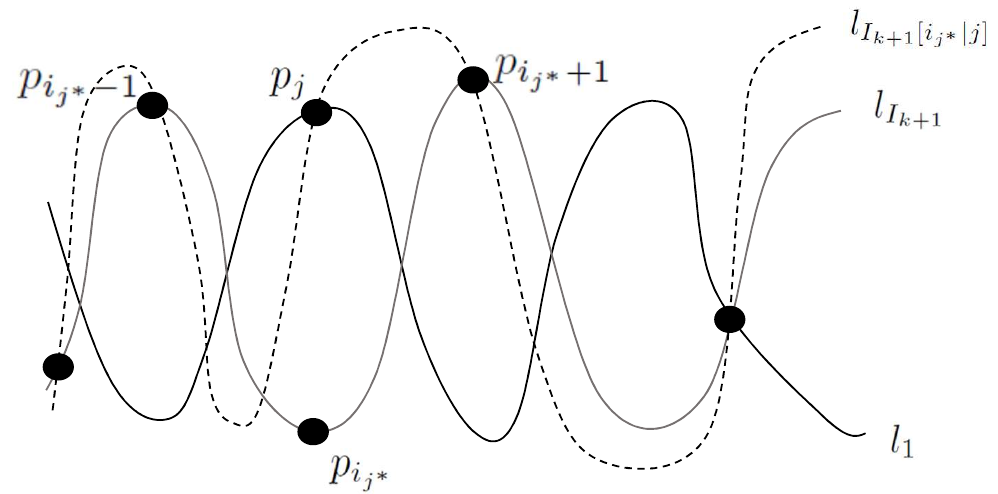}
\end{center}
\caption{$l_1$, $l_{I_{k+1}}$, and $l_{I_{k+1}[i_{j^*} | j]}$ ($p_j, p_{i_{j^*}} \in R^*_{j^*+1}$, $p_j \in \widetilde{P}(l_1)$)}
\label{fig:discussion2_2}
\end{figure}
\noindent
{\bf (Case I-ii)} $(\widetilde{P} \setminus \widetilde{P}(l_{I_{k+1}})) \cap \bigcup_{i=1}^{k+1}{(R'_i \setminus \{ e_l(R'_i), e_r(R'_i) \})} = \emptyset$. 

Let $p_{j_1},\dots,p_{j_{k+1}}$ be the points in $\widetilde{P}(l_1)$ ordered in increasing order of the $x$-coordinates.
Then,  for some $b \in [k]$ it holds that $p_{j_1},\dots,p_{j_b}(=p_{i_{k+1}}) \in \{ e_r(R'_1),\dots,e_r(R'_{k+1}) \}$ and $p_{j_{b+1}},\dots,p_{j_{k+1}} \in R'_{k+2} \setminus \{ e_l(R'_{k+2}) \}$.

First, we assume that $b \leq k-1$.
Let $c$ be the largest integer such that $p_{i_c} \in \widetilde{P}(l_{I_{k+1}}) \setminus \widetilde{P}(l_1)$. Then, we have $c \neq 1$.
The curve $l_{I_{k+1}[i_c \mid j_{b+1}]}$ intersects with $l_1$ and $l_{k+1}$ twice in total in each of the regions $R'_m \setminus \{ e_l(R'_m) \}$ for $m \in [k+1] \setminus \{ 1,c \}$
and once or twice in $R'_1$ (see Figure~\ref{fig:discussion3}).
We also remark that  $l_{I_{k+1}[i_c | j_{b+1}]}$ does not intersect twice with $l_1$ or  with $l_{I_{k+1}}$ on the boundary of  any region mentioned above because otherwise an empty lens is formed.
Therefore, $l_{I_{k+1}[i_c | j_{b+1}]}$ intersects exactly once with $l_1$ and with $l_{I_{k+1}}$ on the boundary of each of the regions $R'_m \setminus \{ e_l(R'_m) \}$ for $m \in [k+1] \setminus \{ 1,c \}$.
 If $l_{I_{k+1}[i_c | j_{b+1}]}$ intersects a total of two times with $l_1$ and $l_{I_{k+1}}$ in $R'_1 \setminus \{ e_r(R'_1)\}$, then it intersects a total of more than $2k$ times with $l_1$ and $l_{I_{k+1}}$, and thus
more than $k$ times with $l_1$ or with $l_{I_{k+1}}$.
This situation cannot be resolved by the empty lens eliminations in the subconfiguration induced by the points on the two curves intersecting more than $k$ times. 
 Since $\widetilde{P}(l_{I_{k+1}[i_c | j_{b+1}]}) \cup \widetilde{P}(l_1) \subset \widetilde{P} \setminus \{ p_{i_c} \}$ and  $\widetilde{P}(l_{I_{k+1}[i_c | j_{b+1}]}) \cup \widetilde{P}(l_{I_{k+1}}) \subset  \widetilde{P} \setminus \{ p_{j_{b+2}} \}$, 
this contradicts the minimality assumption of $\widetilde{P}$.
 If $l_{I_{k+1}[i_c | j_{b+1}]}$ intersects exactly once in total with $l_1$ and $l_{I_{k+1}}$ in $R'_1 \setminus \{ e_r(R'_1)\}$, the curve $l_{I_{k+1}[i_c | j_{b+1}]}$ leaves the lens $R'_1$ at $p_{i_1}$,
and $l_{I_{k+1}}$ lies between $l_1$ and $l_{I_{k+1}[i_c | j_{b+1}]}$ at $x= x(p_{i_1}) + \epsilon$ for sufficiently small $\epsilon > 0$.
Since each two curves of $l_1,l_{I_{k+1}}, l_{I_{k+1}[i_c | j_{b+1}]}$ intersect exactly $k$ times in $H_{(x(e_r(R'_1)),x(e_r(R'_{k+1}))+\epsilon)}$, the situation is the same at $x= x(e_r(R'_{k+1}))+\epsilon$.
Therefore, the curve  $l_{I_{k+1}[i_c | j_{b+1}]}$ must intersect with both $l_1$ and $l_{I_{k+1}}$ in $H_{(x(e_r(R'_{k+1})), x(p_{j_{b+1}}))}$ and thus it intersects more than $k$ times
with $l_1$ or with $l_{I_{k+1}}$ in $\mathbb{R}^2$.  Since this situation cannot be resolved by the empty lens eliminations, it contradicts the minimality assumption of $\widetilde{P}$.

Finally, we consider the case that $b=k$.
Let $p_{i_c} \in \widetilde{P}(l_{I_{k+1}}) \setminus \widetilde{P}(l_1)$. 
Assume that $l_1$ and $l_{I_{k+1}}$ intersect in $H_{(x(e_r(R'_{k+1})), x(p_{j_{k+1}}))}$.
Then, they form a full lens $R'_{k+2}$ and we can apply the empty lens elimination II to $R'_{k+1}$ and $R'_{k+2}$, which is a contradiction.
Therefore, the curves $l_1$ and $l_{I_{k+1}}$ do not intersect in $H_{(x(e_r(R'_{k+1})), x(p_{j_{k+1}}))}$.
Without loss of generality, we assume that $l_1$ is above $l_{I_{k+1}}$ at $x = x(p_{j_{k+1}})$, i.e., $\chi (I_{k+1},j_{k+1}) = +1$.
Then, by considering the above-below relationship of $l_1$ and $l_{I_{k+1}}$ in each lens, the point $p_{i_c}$ is on the $(-1)^{k+2-c}$-side of $l_1$, which implies
$\chi (I_{k+1} \setminus \{ i_c \}, j_{k+1}, i_c) = (-1)^{k+2-c}$.
This contradicts Axiom (B2) of the chirotope axioms.
\\
\begin{figure}[h]
\begin{center}
\includegraphics[scale=0.30, bb = 0 0 500 250]{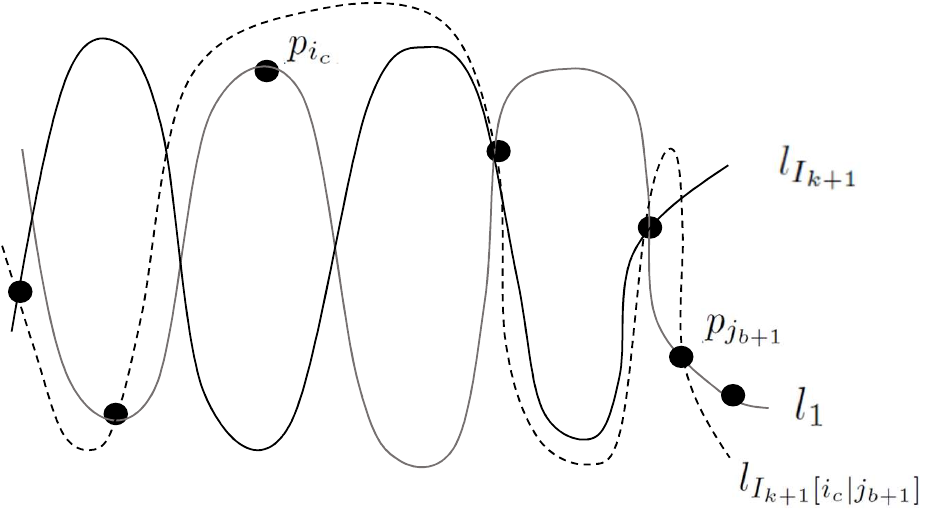} 
\includegraphics[scale=0.30, bb = 0 0 500 250]{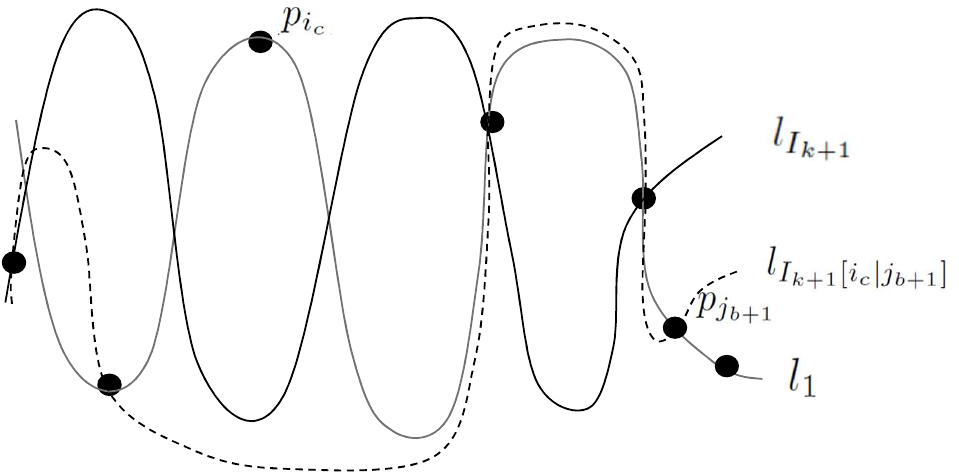} 
\end{center}
\caption{$l_1$, $l_{I_{k+1}}$, and $l_{I_{k+1}[i_c | j_{b+1}]}$}
\label{fig:discussion3}
\end{figure}
\\
 {\bf (Case II)} $l_{I_{k+1}} = l_1$ or $l_{I_{k+1}} = l_2$.
 
We can apply the same discussion as above, by considering $l_1$ and $l_2$ instead of $l_1$ and $l_{I_{k+1}}$.
\end{proof}

\begin{thm}
Let ${\cal M}$ be a degree-$k$  uniform oriented matroid on the ground set $[n]$. Then, there exists a $k$-intersecting pseudoconfiguration of points $PP = ((p_e)_{e \in [n]},L)$
such that ${\cal M} = {\cal M}_{PP}$.
\label{main_thm}
\end{thm}
\begin{proof}
Let $P = ((1,0),(2,0),\dots,(n,0))$.
For each $(i_1,\dots,i_{k+1}) \in \Lambda ([n],k+1)$,
we construct $l_{i_1,\dots,i_{k+1}}$ so that $p_j$ lies above (resp. below) it if  $\chi (i_1,\dots,i_{k+1},j) = +1$ (resp. $-1$)
and $p_{i_1},\dots,p_{i_{k+1}}$ lie on $l_{i_1,\dots,i_{k+1}}$,
and any three curves intersect at the same point $p \notin P$.
Then, apply the empty lens eliminations I and II as far as possible.
\end{proof}

\section{Concluding remarks}
\label{sec:conc}
In this paper, we have introduced a new class of oriented matroids, called degree-$k$ oriented matroids, which abstracts combinatorial behaviors of partitions of  point sets in the plane
by the graphs of polynomial functions of degree $k$.
We have proved that there exists a natural class of geometric objects, called $k$-intersecting pseudoconfigurations of points, that corresponds to degree-$k$ oriented matroids.
This implies that degree-$k$ oriented matroids are a precise combinatorial model for the possible partitions of  point sets in the plane by graphs of polynomial functions of degree $k$.
We have only considered uniform degree-$k$ oriented matroids, but generalization of Proposition~\ref{prop:pp_def_om}  and Theorem~\ref{main_thm} to the non-uniform case is straightforward.

After writing the first version of the paper, the author realized that the axiom of degree-$k$ oriented matroids (in the uniform case) coincides with the axiom of $(k+2)$-signotopes, introduced by Felsner and Weil~\cite{FW01}.
In the context of higher Bruhat orders~\cite{MS89}, Ziegler~\cite{Z93} studied $k$-signotopes under the name of ``consistent subsets'' and gave a geometric interpretation for $k$-signotopes as single element extensions of a cyclic hyperplane arrangement in $\mathbb{R}^{n-k-1}$ (equivalently, as single element liftings of cyclic hyperplane arrangements in $\mathbb{R}^{k}$).
Considering Ziegler's result, our result gives a one-to-one correspondence between (the oriented matroids of) single element liftings of a cyclic arrangement in $\mathbb{R}^k$ and (the equivalence classes of) $k$-intersecting pseudoconfigurations of points.
A special case of this correspondence can be seen in Proposition~\ref{prop:lifting}.

\section*{Acknowledgement}
The author is  grateful to Yuya Hodokuma for the opportunity to read the paper \cite{EM13} together, which motivated the author
to introduce degree-$k$ oriented matroids.
This research was partially supported by JSPS Grant-in-Aid for Young Scientists (B) 26730002.

\end{document}